\documentclass{amsart}
\usepackage{amsmath, amsthm, amscd, amsfonts, amssymb, graphicx, color}
\usepackage[bookmarksnumbered, colorlinks, plainpages]{hyperref}

\newenvironment{rcases}
  {\left.\begin{aligned}}
  {\end{aligned}\right\rbrace}

\input{mathrsfs.sty}

\newtheorem{theorem}{Theorem}[section]
\newtheorem{lemma}[theorem]{Lemma}
\newtheorem{proposition}[theorem]{Proposition}
\newtheorem{corollary}[theorem]{Corollary}
\theoremstyle{definition}
\newtheorem{definition}[theorem]{Definition}
\newtheorem{example}[theorem]{Example}
\theoremstyle{remark}
\newtheorem{remark}[theorem]{Remark}
\begin{document}
\title{ Cyclic $p$-$\varphi$-contraction mappings}

\author[S. M. S. Nabavi Sales]{Sayyed Mohammad Sadegh  Nabavi Sales}
\address{ Department of Mathematics and Computer Science, Hakim Sabzevari University, P.O. Box 397, Sabzevar, Iran.}
\email[S. M. S. Nabavi Sales]{sadegh.nabavi@gmail.com; sadegh.nabavi@hsu.ac.ir}
%\thanks{$^*$Corresponding author}
\subjclass[2020]{Primary 47H10, 41A65; Secondary 41A52}

\date{}


\keywords{Fixed point, Cyclic $p$-$\varphi$-contraction mapping, metric space, Banach space, Reflexive Banach space}

\begin{abstract}
We introduce a new class of mappings called cyclic $p$-$\varphi$-contraction mappings and investigate the existence and uniqueness of fixed point for such mappings defined on metric spaces, uniformly convex Banach spaces, or reflexive Banach spaces.
\end{abstract}

 \maketitle
\section{Introduction  and Preliminaries}
The well known Banach fixed point theorem, also known as the contraction mapping theorem, is an important result in the theory of complete metric spaces. It discusses the existence of the unique fixed point of a contraction mapping defined on a complete metric space. The contractive condition for a mapping, considered in its classical version, seems to be somewhat strong, since it indeed makes such a mapping continuous. The investigations in the field, since the time of the proof of the original version of the Banach fixed point theorem, show that by relaxing the contractive condition we can provide examples of mappings that are not necessarily continuous and for which the existence and uniqueness of fixed point is guaranteed; see \cite{Ki}. Some recent work in this area that motivated this article include: \cite{ThSh,ElVe,KiSrVe}. In \cite{KiSrVe}, the authors consider mapping $T:X\to X$ that is cyclic contraction meaning that, there are nonempty subsets $A_1,\cdots,A_m$ of a metric space $(X,d)$, $T(A_i)\subset A_{i+1}$ for $i=1,2,\cdots,m$ with $A_{m+1}=A_1$, and there exists $0<\alpha<1$ so that $d(Tx,Ty)\leq\alpha d(x,y)$ for $x\in A_i$ and $y\in A_{i+1}$. It is proved there that if $X$ is complete and $A_i$s are closed, then $T$ has a unique fixed point in $\cap_{i=1}^mA_i$. Note that the contractive condition considered in this result does not make $T$ continuous, so it is actually an important improvement. This result is improved in \cite{ThSh,ElVe} considering cyclic $\alpha$-contraction mappings and cyclic $\varphi$-contraction mappings, where $0<\alpha<1$ and $\varphi:[0,\infty)\to[0,\infty)$ is a strictly increasing function. For nonempty subsets $A_1$ and $A_2$ of a metric space $(X,d)$, a mapping $T:A_1\cup A_2\to A_1\cup A_2$ is said to be cyclic $\varphi$-contraction mapping if $T(A_1)\subset A_2$, $T(A_2)\subset A_1$ and $$d(Tx,Ty)\leq d(x,y)-\varphi(d(x,y))+\varphi(d(A_1,A_2)),\qquad(x\in A_1,y\in A_2),$$ where $d(A_1,A_2)$ stands for the distance between $A_1$ and $A_2$. If $\varphi(t)=\alpha t$, then a cyclic $\varphi$-contraction mapping is called a cyclic $\alpha$-contraction mapping. In\cite{ThSh,ElVe}, the authors prove some fixed point theorems in general metric spaces, uniformly convex Banach spaces, and reflexive spaces. In this note, we generalize their method in such spaces for mappings which we call cyclic $p$-$\varphi$-contraction, where $p\geq1$ and $\varphi:[0,\infty)\to[0,\infty)$ is a strictly increasing function. Now, some preliminaries.

Let $X$ be a Banach space. $X$ is said to be uniformly convex if there exists a strictly increasing function $\delta:(0,2]\to[0,1]$ such that the following implication holds for all $x,y,w\in X, R>0$ and $r\in[0,2R]$:
\begin{eqnarray*}
\begin{rcases}
  \|x-w\|\leq R \\
 \|y-w\|\leq R \\
  \|x-y\|\geq r
\end{rcases}\Rightarrow\left\|{{x+y}\over{2}}-w\right\|\leq\left(1-\delta({r\over R})\right)R.
\end{eqnarray*}
The space $X$ is called strictly convex if the following implication holds for all $x,y,w\in X$ and $R>0$:
\begin{eqnarray*}
\begin{rcases}
  \|x-w\|\leq R \\
 \|y-w\|\leq R \\
  x\neq y
\end{rcases}\Rightarrow\left\|{{x+y}\over{2}}-w\right\|< R.
\end{eqnarray*}
If $A\subset X$, by $A-A$ we mean the set $\{x-y;x,y\in A\}$. A metric space $X$ is said to be boundedly compact if every bounded sequence in $X$ has a convergent subsequence.
\subsection{Notion of cyclic $p$-$\varphi$-contraction mapping}
\begin{definition}
Let $A_1$, $A_2,\cdots, A_m$ be nonempty subsets of a metric space $(X,d)$. For $1\leq p\leq\infty$ we define $d_p(A_1,A_2,\cdots,A_m)$ to be $$d_p(A_1,A_2,\cdots,A_m):=\left(\sum_{i=1}^md(A_i,A_{i+1})^p\right)^{1\over p},$$ if $1\leq p<\infty$ and $$d_\infty(A_1,A_2,\cdots,A_m):=\max\{d(A_1,A_2),d(A_2,A_3),\cdots,d(A_m,A_1)\}.$$ Note that in this notation, we in fact regard $A_1$, $A_2,\cdots, A_m$ as a chain of subsets of $X$. For $x_1,\cdots,x_m,y_1,\cdots,y_m\in X$, $d_p(x_1,\cdots,x_m;y_1,\cdots,y_m)$ is defined as follows:
for $1\leq p<\infty$
\begin{eqnarray*}
d_p(x_1,\cdots,x_m;y_1,\cdots,y_m):=\left(\sum_{i=1}^md(x_i,y_{i+1})^p\right)^{1\over p}
\end{eqnarray*}
where $y_{m+1}=y_1$, and
\begin{eqnarray*}
d_\infty(x_1,\cdots,x_m;y_1,\cdots,y_m):=\max\{d(x_1,y_2),d(x_2,y_3),\cdots,d(x_m,y_1)\}.
\end{eqnarray*}
We agree to use the notation $$d_p(x_1,\cdots,x_m)=d_p(x_1,\cdots,x_m;x_1,\cdots,x_m).$$
A mapping $T:X\to X$ is said to be cyclic if $T(A_i)\subset A_{i+1}$ for $i=1,2,\cdots,m$ with $A_{m+1}=A_1$. For a strictly increasing function  $\varphi:[0,\infty)\to[0,\infty)$, and  $1\leq p\leq\infty$, $T$ is said to be cyclic $p$-$\varphi$-contraction if it is cyclic and satisfies
\begin{eqnarray}\label{CC}
d_p(Tx_1,\cdots,Tx_m;Ty_1,\cdots,Ty_m)&\leq& d_p(x_1,\cdots,x_m;y_1,\cdots,y_m)\nonumber\\&&-\varphi(d_p(x_1,\cdots,x_m;y_1,\cdots,y_m))\nonumber\\&&+\varphi(d_p(A_1,\cdots,A_m))
\end{eqnarray}
 for all $x_i,y_i\in A_i$ where $i=1,2,\cdots,m$. For $\alpha\in(0,1)$, if $\varphi$ is defined to be $\varphi(t)=\alpha t$, then a cyclic $p$-$\varphi$-contraction is called cyclic $p$-$\alpha$-contraction.
\end{definition}
In what follows, we assume that $\varphi:[0,\infty)\to[0,\infty)$ is a strictly increasing function, $1\leq p\leq\infty$, $A_1$, $A_2,\cdots, A_m$ are nonempty subsets of a metric space $(X,d)$, and $T:\cup_{i=1}^mA_i\to\cup_{i=1}^mA_i$ is a cyclic $p$-$\varphi$-contraction, which means that for $i=1,\cdots,m$ we have that $T(A_i)\subset A_{i+1}$ with $A_{m+1}=A_1$ and $T$ satisfies \eqref{CC}.
\begin{remark}
Note that, for a cyclic $p$-$\varphi$-contraction $T:\cup_{i=1}^mA_i\to\cup_{i=1}^mA_i$ if $x_i\in A_i$ for $i=1,\cdots,m$, then obviously from \eqref{CC} we have that \begin{eqnarray}\label{FP8}d_p(Tx_1,\cdots,Tx_m;Ty_1,\cdots,Ty_m)\leq d_p(x_1,\cdots,x_m;y_1,\cdots,y_m).\end{eqnarray} Thus, if $x_0\in A_1$ and $x_n:=Tx_{n-1}$, then for all natural numbers $n$ and $k$ we have that
\begin{eqnarray*}d_p(x_n,x_{n+1},\cdots,x_{n+m-1};x_k,x_{k+1},\cdots,x_{k+m-1})\leq\\ d_p(x_{n-1},x_n,\cdots,x_{n+m-2};x_{k-1},x_k,\cdots,x_{k+m-2}),\end{eqnarray*} which obviously implies that
$$d_p(x_n,x_{n+1},\cdots,x_{n+m-1})\leq d_p(x_{n-1},x_n,\cdots,x_{n+m-2}).$$
\end{remark}
Throughout the paper we assume that $x_0\in A_1$ and $x_n=Tx_{n-1}$ where $T$ is a cyclic $p$-$\varphi$-contraction.
\section{General metric spaces}
In this section we state some fixed point theorems for mappings defined on metric spaces. The results of this section are mainly the generalization of those of \cite{ElVe} expressed for metric spaces. Here, we assume that $(X,d)$ is an abstract metric space and $A_1,\cdots,A_m\subset X$ and $T:\cup_{i=1}^mA_i\to\cup_{i=1}^mA_i$ is a cyclic $p$-$\varphi$-contraction mapping and $x_0\in A_1$ and $x_n:=Tx_{n-1}$ for all natural numbers $n$,
unless we intend another assumption, in which case we state it explicitly.
\begin{theorem}\label{XX}
 $$d_p(x_n, x_{n+1},\cdots,x_{n+m-1})\to d_p(A_1,A_2,\cdots,A_m)$$ as $n\to\infty$.
\end{theorem}
\begin{proof}We can follow the method of the proof of \cite[Theorem 3]{ThSh} to come to a proof for this theorem. We omit details.
\end{proof}
\begin{theorem}\label{SS1}
 If $\{x_{mn}\}_{n\geq0}$ has a convergent subsequence in $A_1$, then there exists $x\in A_1$ so that $$d_p(x,Tx,\cdots,T^{m-1}x)=d_p(A_1,A_2,\cdots,A_m).$$
\end{theorem}
\begin{proof}
We may assume that  $\{x_{mn}\}_{n\geq0}$ is convergent and $x_{mn}\to x$ as $n\to\infty$. First assume that $p=\infty$. Therefore,
\begin{eqnarray*}
d_\infty(A_1,\cdots,A_m)&\leq& d_\infty(x,x_{mn+1},x_{mn+2},\cdots,x_{mn+m-1})\\
&=&\max\{d(x,x_{mn+1}),d(x_{mn+1},x_{mn+2}),\cdots,d(x_{mn+m-1},x)\}\\
&\leq&\max\{d(x,x_{mn})+d(x_{mn},x_{mn+1}),d(x,x_{mn})+d(x_{mn+1},x_{mn+2})\\&&,\cdots,d(x,x_{mn})+d(x_{mn+m-1},x_{mn})\}\\&=&d(x,x_{mn})+d_\infty(x_{mn},x_{mn+1},\cdots,x_{mn+m-1})\to d_\infty(A_1,\cdots,A_m).
\end{eqnarray*}
It follows that
\begin{eqnarray*}
d_\infty(A_1,\cdots,A_m)&\leq& d_\infty(x_{mn},Tx,x_{mn+2},\cdots,x_{mn+m-1})\\
&\leq&d_\infty(x_{mn-1},x,x_{mn+1}\cdots,x_{mn+m-2})\to d_\infty(A_1,\cdots,A_m).
\end{eqnarray*}
Hence $d_\infty(x_{mn},Tx,x_{mn+2},\cdots,x_{mn+m-1})\to d_\infty(A_1,\cdots,A_m)$ which implies that
\begin{eqnarray*}
d_\infty(A_1,\cdots,A_m)&\leq& d_\infty(x,Tx,x_{mn+2},\cdots,x_{mn+m-1})\\
&\leq&d(x,x_{mn})+d_\infty(x_{mn},Tx,x_{mn+1}\cdots,x_{mn+m-1})\\&&\to d_\infty(A_1,\cdots,A_m)
\end{eqnarray*}
that means $d_\infty(x,Tx,x_{mn+2},\cdots,x_{mn+m-1})\to d_\infty(A_1,\cdots,A_m)$. So, by an easy induction we can show that
$$d_\infty(x_{mn},Tx,T^2x,\cdots,T^{m-2}x,T^{m-1}x)\to d_\infty(A_1,\cdots,A_m).$$
This yeilds that
\begin{eqnarray*}
d_\infty(A_1,\cdots,A_m)&\leq&d_\infty(x,Tx,T^2x,\cdots,T^{m-2}x,T^{m-1}x)\\
&\leq&d(x,x_{mn})+d_\infty(x_{mn},Tx,T^2x,\cdots,T^{m-2}x,T^{m-1}x)\\&&\to d_\infty(A_1,\cdots,A_m).
\end{eqnarray*}
Now, assume that $1\leq p<\infty$. We have that
\begin{eqnarray*}
d_p(A_1,\cdots,A_m)&\leq& d_p(x,x_{mn+1},x_{mn+2},\cdots,x_{mn+m-1})\\
&=&\{d(x,x_{mn+1})^p+d(x_{mn+1},x_{mn+2})^p+\cdots+d(x_{mn+m-1},x)^p\}^{1\over p}\\
&\leq&\{(d(x,x_{mn})+d(x_{mn},x_{mn+1}))^p+d(x_{mn+1},x_{mn+2})^p\\&&+\cdots+(d(x,x_{mn})+d(x_{mn+m-1},x_{mn}))^p\}^{1\over p}\\&\leq&2^{1\over p}d(x,x_{mn})+d_p(x_{mn},x_{mn+1},\cdots,x_{mn+m-1})\to d_p(A_1,\cdots,A_m),
\end{eqnarray*}
which implies that $d_p(x,x_{mn+1},x_{mn+2},\cdots,x_{mn+m-1})\to d_p(A_1,\cdots,A_m)$.
Thus
\begin{eqnarray*}
d_p(A_1,\cdots,A_m)&\leq& d_p(x,Tx,x_{mn+2},\cdots,x_{mn+m-1})\\
&\leq&2^{1\over p}d(x,x_{mn})+d_p(x_{mn},Tx,x_{mn+1}\cdots,x_{mn+m-2})\\&&\to d_p(A_1,\cdots,A_m),
\end{eqnarray*}
that means $ d_p(x,Tx,x_{mn+2},\cdots,x_{mn+m-1})\to d_p(A_1,\cdots,A_m).$ An easy induction shows that $$d_p(x_{mn},Tx,T^2x,\cdots,T^{m-2}x,T^{m-1}x)\to d_p(A_1,\cdots,A_m)$$
and this yeilds that
\begin{eqnarray*}
d_p(A_1,\cdots,A_m)&\leq&d_p(x,Tx,T^2x,\cdots,T^{m-2}x,T^{m-1}x)\\
&\leq&2^{1\over p}d(x,x_{mn})+d_p(x_{mn},Tx,T^2x,\cdots,T^{m-2}x,T^{m-1}x)\\&&\to d_p(A_1,\cdots,A_m).
\end{eqnarray*}
\end{proof}
\begin{remark}In the same way as the proof of the previous theorem, we could show that if $\{x_{mn+i}\}_{n\geq0}$ has a convergent subsequence in $A_{i+1}$, then there exists an $x\in A_{i+1}$ so that $$d_p(x,Tx,\cdots,T^{m-1}x)=d_p(A_1,A_2,\cdots,A_m).$$
\end{remark}
\begin{theorem}\label{MET1}
Let $0<\alpha^m<{1\over{2^{1\over p}}}$ and let $T$ be a cyclic $p$-$\alpha$-contraction. Then sequences $\{x_{mn}\}$, $\{x_{mn+1}\}$, $\cdots$, $\{x_{mn+m-1}\}$ are all bounded.
\end{theorem}
\begin{proof}
We show that $\{x_{mn}\}$ is bounded. For the others the method is similar. To the contrary assume that $\{x_{mn}\}$ is not bounded which means that for any natural number $M$ there exists a natural number $N$ so that
$$d(T^{mN-m}x_0,T^{m+1}x_0)\leq M\:{\rm but}\:d(T^{mN}x_0,T^{m+1}x_0)>M.$$
Thus
\begin{eqnarray*}
M&<&d(T^{mN}x_0,T^{m+1}x_0)\\&\leq& d_p(T^{mN}x_0,T^{m+1}x_0,T^{m+2}x_0,\cdots,T^{2m-1}x_0)\\&\leq&\alpha^md_p(T^{mN-m}x_0,Tx_0,T^2x_0,\cdots,T^{m-1}x_0)+(1-\alpha^m)d_p(A_1,\cdots,A_m),
\end{eqnarray*}
which implies that
\begin{eqnarray*}
{M\over{\alpha^m}}-{{1-\alpha^m}\over{\alpha^m}}d_p(A_1,\cdots,A_m)&<&d_p(T^{mN-m}x_0,Tx_0,T^2x_0,\cdots,T^{m-1}x_0)\\
&\leq&d(T^{mN-m}x_0,T^{m+1}x_0)+d(T^{m+1}x_0,x_0)\\&&+d_p(x_0,Tx_0,\cdots,T^{m-1}x_0)\\&\leq&M+d(T^{m+1}x_0,x_0)+d_p(x_0,Tx_0,\cdots,T^{m-1}x_0).
\end{eqnarray*}
Thus
\begin{eqnarray*}
\left({{1}\over{\alpha^m}}-2^{1\over p}\right)M&\leq&d(T^{m+1}x_0,x_0)+d_p(x_0,Tx_0,\cdots,T^{m-1}x_0)\\&&+{{1-\alpha^m}\over{\alpha^m}}d_p(A_1,\cdots,A_m).
\end{eqnarray*}
According to our assumption about $\alpha$ this implies that
\begin{eqnarray*}
M\leq {{2^{1\over p}d(T^{m+1}x_0,x_0)+d_p(x_0,Tx_0,\cdots,T^{m-1}x_0)+{{1-\alpha^m}\over{\alpha^m}}d_p(A_1,\cdots,A_m)}\over{\left({{1}\over{\alpha^m}}-2^{1\over p}\right)}}
\end{eqnarray*}
which is impossible, because $M$ is an arbitrary natural number, but the terms in the right side of the inequality are all constants.
\end{proof}
Notice that, in the recent theorem if $p=\infty$ then $\alpha$ can be chosen to be in $(0,1)$.
\begin{theorem}Let $0<\alpha^m<{1\over{2^{1\over p}}}$ and let $T$ be a cyclic $p$-$\alpha$-contraction. If at least one of $A_i$s is boundedly compact, then there exists $x\in\cup_{i=1}^mA_i$ so that
$$d_p(x,Tx,\cdots,T^{m-1}x)=d_p(A_1,A_2,\cdots,A_m).$$
\end{theorem}
\begin{corollary}\label{C1}
Let $A_1$, $A_2,\cdots, A_m$ be nonempty closed subsets of a normed space $X$, and let $0<\alpha^m<{1\over{2^{1\over p}}}$ and let $T$ be a cyclic $p$-$\alpha$-contraction. If at least the linear span of one of $A_i$s is a finite dimensional subspace of $X$, then there exists $x\in\cup_{i=1}^mA_i$ so that
$$d_p(x,Tx,\cdots,T^{m-1}x)=d_p(A_1,A_2,\cdots,A_m).$$
\end{corollary}
\begin{example}
Let $\alpha$ be a number satisfying the condition of Corollary \ref{C1}. For $i=1,\cdots,m$ and $1\leq q\leq\infty$, consider $A_i\subset \ell^q$ to be defined by $A_i=\{(1+\alpha^{mn+i-1})e_{mn+i-1}\}_{n\in\mathbb{N}}$. Suppose that
$$T((1+\alpha^{mn+i-1})e_{mn+i-1})=(1+\alpha^{mn+i})e_{mn+i},\qquad(n\in\mathbb{N}).$$
We are now going to check that $T$ is cyclic $p$-$\alpha$-contraction. Assume that $p=\infty$ and $1\leq q\leq\infty$. We have that

$$d_\infty(T((1+\alpha^{mn_1})e_{mn_1}),T((1+\alpha^{mn_2+1})e_{mn_2+1}),\cdots,T((1+\alpha^{mn_m+m-1})e_{mn_m+m-1}))$$
\begin{eqnarray*}
&=&\max\{\|T((1+\alpha^{mn_1})e_{mn_1})-T((1+\alpha^{mn_2+1})e_{mn_2+1})\|_q\\&&,\|T((1+\alpha^{mn_2+1})e_{mn_2+1})-T((1+\alpha^{mn_3+2})e_{mn_3+2})\|_q\\&&,\cdots,\|T((1+\alpha^{mn_{m}+m-1})e_{mn_{m}+m-1})-T((1+\alpha^{mn_1+})e_{mn_1})\|_q\}\\
&=&\max\{[(1+\alpha^{mn_1+1})^q+(1+\alpha^{mn_2+2})^q]^{1\over q},\cdots,[(1+\alpha^{mn_m+m})^q+(1+\alpha^{mn_1+1})^q]^{1\over q}\}\\
&=&\max\{[(\alpha+\alpha^{mn_1+1}+(1-\alpha))^q+(\alpha+\alpha^{mn_2+2}+(1-\alpha))^q]^{1\over q}\\&&,\cdots,[(\alpha+\alpha^{mn_m+m}+(1-\alpha))^q+(\alpha+\alpha^{mn_1+1}+(1-\alpha))^q]^{1\over q}\}\\
&\leq&\max\{\alpha[(1+\alpha^{mn_1+1})^q+(1+\alpha^{mn_2+2}+)^q]^{1\over q}+2^{1\over q}(1-\alpha)\\&&,\cdots,[(1+\alpha^{mn_m+m}+)^q+(1+\alpha^{mn_1+1})^q]^{1\over q}+2^{1\over q}(1-\alpha)\}\\&
=&\alpha\max\{[(1+\alpha^{mn_1+1})^q+(1+\alpha^{mn_2+2}+)^q]^{1\over q}\\&&,\cdots,[(1+\alpha^{mn_m+m}+)^q+(1+\alpha^{mn_1+1})^q]^{1\over q}\}+2^{1\over q}(1-\alpha).
\end{eqnarray*}
Similarly for the case of $1\leq p<\infty$ we can obtain that
$$d_p(T((1+\alpha^{mn_1})e_{mn_1}),T((1+\alpha^{mn_2+1})e_{mn_2+1}),\cdots,T((1+\alpha^{mn_m+m-1})e_{mn_m+m-1}))$$
\begin{eqnarray*}
&=&\{\|T((1+\alpha^{mn_1})e_{mn_1})-T((1+\alpha^{mn_2+1})e_{mn_2+1})\|_q^p\\&&+\|T((1+\alpha^{mn_2+1})e_{mn_2+1})-T((1+\alpha^{mn_3+2})e_{mn_3+2})\|_q^p\\&&+\cdots+\|T((1+\alpha^{mn_{m}+m-1})e_{mn_{m}+m-1})-T((1+\alpha^{mn_1+})e_{mn_1})\|_q^p\}\\&\leq&\alpha\{[(1+\alpha^{mn_1+1})^q+(1+\alpha^{mn_2+2}+)^q]^{p\over q}\\&&+\cdots+[(1+\alpha^{mn_m+m}+)^q+(1+\alpha^{mn_1+1})^q]^{p\over q}\}^{1\over p}+m^{1\over p}2^{1\over q}(1-\alpha).
\end{eqnarray*}
Notice that there dose not exist an $x\in\cup_{i=1}^mA_i$ so that
$$d_p(x,Tx,\cdots,T^{m-1}x)=d_p(A_1,A_2,\cdots,A_m),$$ so the Corollary \ref{C1} fails if we remove the condition of finiteness of dimension for at least one of $A_i$s.
\end{example}
\begin{lemma}\label{V1}${}$

1)  If $p=\infty$, then there exists $1\leq i\leq m$ so that $d(x_{mn+i-1},x_{mn+i})\to d(A_i,A_{i+1})$ as $n\to\infty$.

2) If $1\leq p<\infty$, then for any $1\leq i\leq m$ we have that $d(x_{mn+i-1},x_{mn+i})\to d(A_i,A_{i+1})$ as $n\to\infty$.
\end{lemma}
\begin{proof}
1)  Assume that $ d(A_i,A_{i+1})=d_\infty(A_1,A_2,\cdots,A_m)$. Thus
$$d(A_i,A_{i+1})\leq d(x_{mn+i-1},x_{mn+i})\leq d_\infty(x_{mn},x_{mn+1},x_{mn+2}\cdots,x_{mn+m-1}).$$ According to Theorem \ref{XX}, $d_\infty(x_n, x_{n+1},\cdots,x_{n+m-1})\to d_\infty(A_1,A_2,\cdots,A_m)$ as $n\to\infty$, which implies that $d(x_{mn+i-1},x_{mn+i})\to d(A_i,A_{i+1})$ as $n\to\infty$.

2) Let $1\leq i\leq m$. If $d(x_{mn+i-1},x_{mn+i})\not\to d(A_i,A_{i+1})$, then there exists $\varepsilon_0>0$ so that for each $N\in\mathbb{N}$, there exists $k\geq N$ so that $d(x_{mk+i-1},x_{mk+i})-d(A_i,A_{i+1})\geq\varepsilon_0$ that is $d(x_{mk+i-1},x_{mk+i})\geq d(A_i,A_{i+1})+\varepsilon_0$. Hence
\begin{eqnarray*}
0&<&\{d(A_1,,A_2)^p+\cdots+(d(A_{i},A_{i+1})+\varepsilon_0)^p+\cdots+d(A_{m-1},A_m)^p\}^{1\over p}\\&&-\{d(A_1,,A_2)^p+\cdots+d(A_{i},A_{i+1})^p+\cdots+d(A_{m-1},A_m)^p\}^{1\over p}\\
&\leq&\{d(x_{mk},x_{mk+1})^p+\cdots+d(x_{mk+i-1},x_{mk+i})^p+\cdots+d(x_{mk+m-1},x_{mk})^p\}^{1\over p}\\&&-\{d(A_1,,A_2)^p+\cdots+d(A_{i},A_{i+1})^p+\cdots+d(A_{m-1},A_m)^p\}^{1\over p},
\end{eqnarray*}
which contradicts $d_p(x_{mn}, x_{mn+1},\cdots,x_{mn+m-1})\to d_p(A_1,A_2,\cdots,A_m).$
\end{proof}
Finally, we assume that $d_p(A_1,A_2,\cdots,A_m)=0$. This leads to the following theorem, which is, in some sense, close to the famous Banach fixed point theorem.
\begin{theorem}\label{BANACH}
Let $T:\cup_{i=1}^mA_i\to\cup_{i=1}^mA_i$ be a cyclic $p$-$(1-\alpha)$-contraction for some $0<\alpha<1$, and let $d_p(A_1,A_2,\cdots,A_m)=0$. If $X$ is complete and $A_i$s are closed subsets of $X$, then $T$ has a unique fixed point $x\in\cap_{i=1}^mA_i$.
\end{theorem}
\begin{proof}Note that the assumption implies that $d(A_i,A_{i+1})=0$ for $i=1,\cdots,m$. For a proof we follow the method of the original proof of the Banach fixed point theorem. From the definition we have that
 \begin{eqnarray*}&&d_p(x_{mn},x_{mn+1},\cdots,x_{mn+m-1};x_{mn-m},x_{mn-m+1},\cdots,x_{mn-1})\\&\leq& \alpha d_p(x_{mn-1},x_{mn},\cdots,x_{mn+m-2};x_{mn-m-1},x_{mn-m},\cdots,x_{mn-2})\\&\leq& \alpha^2d_p(x_{mn-2},x_{mn-1},\cdots,x_{mn+m-3};x_{mn-m-2},x_{mn-m-1},\cdots,x_{mn-3})\\&.&\qquad\qquad\qquad\qquad\qquad\qquad\qquad\\&.&\qquad\qquad\qquad\qquad\qquad\qquad\qquad\\&.&\qquad\qquad\qquad\qquad\qquad\qquad\qquad\\&\leq& \alpha^{mn-m}d_p(x_{m},x_{m+1},\cdots,x_{2m-1};x_{0},x_{1},\cdots,x_{m-1}).\qquad
\end{eqnarray*}
Hence for positive integers $n$ and $k$ with $n\geq k$, obviously we have that
\begin{eqnarray*}
&&d_p(x_{mn},x_{mn+1},\cdots,x_{mn+m-1};x_{mk},x_{mk+1},\cdots,x_{mk+m-1})\\&\leq&d_p(x_{mn},x_{mn+1},\cdots,x_{mn+m-1};x_{mn-m},x_{mn-m+1},\cdots,x_{mn-1})\\&&+d_p(x_{mn-m},x_{mn-m+1},\cdots,x_{mn-1};x_{mn-2m},x_{mn-2m+1},\cdots,x_{mn-m-1})\\&&.\\&&.\\&&.\\&&+d_p(x_{mk+m},x_{mk+m+1},\cdots,x_{mk+2m-1};x_{mk},x_{mk+1},\cdots,x_{mk+m-1})\\&\leq&\left(\alpha^{mn-m}+\alpha^{mn-2m}+\cdots+\alpha^{mk}\right)\\&&\times d_p(x_{m},x_{m+1},\cdots,x_{2m-1};x_{0},x_{1},\cdots,x_{m-1})\\&=&\alpha^{mk}\left({{1-\alpha^{mn-mk}}\over{1-\alpha}}\right)d_p(x_{m},x_{m+1},\cdots,x_{2m-1};x_{0},x_{1},\cdots,x_{m-1})\\&\leq&\left({{\alpha^{mk}}\over{1-\alpha}}\right)d_p(x_{m},x_{m+1},\cdots,x_{2m-1};x_{0},x_{1},\cdots,x_{m-1}).
\end{eqnarray*}
Thus, since $0<\alpha<1$, for a given $\varepsilon>0$ we can find a positive integer $N$ so that for $k\geq N$
\begin{eqnarray*}
\left({{\alpha^{mk}}\over{1-\alpha}}\right)d_p(x_{m},x_{m+1},\cdots,x_{2m-1};x_{0},x_{1},\cdots,x_{m-1})<\varepsilon.
\end{eqnarray*}
 Therefore, for any $\varepsilon>0$ there exists a positive integer $N$ so that if $n\geq k\geq N$, then
\begin{eqnarray}\label{BAN1}
d_p(x_{mn},x_{mn+1},\cdots,x_{mn+m-1};x_{mk},x_{mk+1},\cdots,x_{mk+m-1})<\varepsilon.
\end{eqnarray}
Now, we claim that

($\ast$) for $i=1,\cdots,m$ and for $\varepsilon>0$ there exists a positive integer $N$ such that if $n>k\geq N$ then $d(x_{mn+i-1},x_{mk+i})<\varepsilon.$

Assume that ($\ast$) is not true for some $i\in\{1,\cdots,m\}$. Hence there exists an $\varepsilon_0$ so that for any positive integer $t$ there exist $n_t\geq k_t\geq t$ such that $d(x_{mn_t+i-1},x_{mk_t+i})\geq\varepsilon_0.$ But by \eqref{BAN1}  there exists a positive integer $N$ so that if $n\geq k\geq N$, then
\begin{eqnarray*}
d_p(x_{mn},x_{mn+1},\cdots,x_{mn+m-1};x_{mk},x_{mk+1},\cdots,x_{mk+m-1})<\varepsilon_0.
\end{eqnarray*}
Pick $n_N\geq k_N\geq N$ we have that
\begin{eqnarray*}
\varepsilon_0&\leq& d(x_{mn_N+i-1},x_{mk_N+i})\\&\leq& d_p(x_{mn_N},x_{mn_N+1},\cdots,x_{mn_N+m-1};x_{mk_N},x_{mk_N+1},\cdots,x_{mk_N+m-1})\\&<&\varepsilon_0;
\end{eqnarray*} a contradiction that ensures the validity of ($\ast$). But,
($\ast$) implies that $\{x_{mn+i-1}\}_{n\in\mathbb{N}}$ is in fact a Cauchy sequence in $A_i$, because of
$$d(x_{mn+i-1},x_{mk+i-1})\leq d(x_{mn+i-1},x_{mk+i})+d(x_{mk+i-1},x_{mk+i}).$$
 Since $X$ is complete and $A_i$ is assumed to be closed in $X$, $\{x_{mn+i-1}\}$ converges to an element of $A_i$, $\hat{x}_i$ says. On the other hand $$d(\hat{x}_i,x_{mn+i})\leq d(\hat{x}_i,x_{mn+i-1})+d(x_{mn+i-1},x_{mn+i})$$ which ensures that $\hat{x}_i=\hat{x}_{i+1}$ for $i=1,\cdots,m$. Thus $\{x_n\}$ is convergent to an $x$ in $\cap_{i=1}^mA_i\neq\emptyset$.

 To show that $x$ is a fixed point for $T$, we note that
 \begin{eqnarray*}
 d(x,Tx)&\leq& d(x_n,x)+d(Tx_{n-1},Tx)\\&\leq& d(x_n,x)+d_p(Tx_{n-1},Tx,Tx,\cdots,Tx)\\&\leq& d(x,x_n)+\alpha d_p(x_{n-1},x,x,\cdots,x).
 \end{eqnarray*}But $x_n\to x$ as $n\to\infty$, that implies that $d(x,Tx)=0$ i.e. $x$ is a fixed point for $T$. Now, if $y$ is another fixed point for $T$, that means $Ty=y$, then
 \begin{eqnarray*}d(x,y)&=&d(Tx,Ty)\\&=& {{1}\over{2^{{1}\over{p}}}}d_p(Tx,Ty,Tx,Tx,\cdots,Tx)\\&\leq&\alpha {{1}\over{2^{{1}\over{p}}}}d_p(x,y,x,\cdots,x)\\&=&\alpha d(x,y).\end{eqnarray*} Since $0<\alpha<1$, this implies that $d(x,y)=0$ which establishes the uniqueness of the fixed point of $T$.
 \end{proof}
\section{Uniformly convex Banach spaces}
In this section we state some fixed point theorems for mappings defined on uniformly convex Banach spaces. The results of this section are mainly the generalization of those of \cite[Section 2]{ThSh}. Here, we assume that $(X,\|\cdot\|)$ is a uniformly convex Banach space and $A_1,\cdots,A_m\subset X$ and $T:\cup_{i=1}^mA_i\to\cup_{i=1}^mA_i$ is a cyclic $p$-$\varphi$-contraction mapping and $x_0\in A_1$ and $x_n:=Tx_{n-1}$ for all natural numbers $n$,
unless we have another assumption in our mind, in which case we state it explicitly.
\begin{proposition}\label{VV}Let $A_{i_0}$ be convex when $d(A_{i_0},A_{i_0+1})=d_\infty(A_1,\cdots,A_m)$. Assume that $T$ is a cyclic $\infty$-$\varphi$-contraction mapping for some strictly increasing function $\varphi:[0,\infty)\to[0,\infty)$. Then $\|x_{mn+i_0-1}-x_{mn+m+i_0-1}\|\to0$ and $\|x_{mn+i_0}-x_{mn+m+i_0}\|\to0$.
\end{proposition}
\begin{proof}
Assume that $\|x_{mn+i_0-1}-x_{mn+m+i_0-1}\|\not\to0$. Thus there exists an $\varepsilon_0>0$ so that for any $k\geq0$ there exists $n_k\geq k$ so that \begin{eqnarray}\label{V2}\|x_{mn_{n_k}+i_0-1}-x_{mn_{n_k}+m+i_0-1}\|\geq\varepsilon_0.\end{eqnarray}Pick $\gamma>0$ that meets ${{\varepsilon_0}\over{\gamma}}>d(A_{i_0},A_{i_0+1})$ and choose an $\varepsilon>0$ satisfying
\begin{eqnarray*}
0<\varepsilon<\min\left\{{{\varepsilon_0}\over{\gamma}}-d(A_{i_0},A_{i_0+1}),{{d(A_{i_0},A_{i_0+1})\delta(\gamma)}\over{1-\delta(\gamma)}}\right\}.
\end{eqnarray*}

According to Lemma \ref{V1} $\|x_{mn+i_0-1}-x_{mn+i_0}\|\to d(A_{i_0},A_{i_0+1})$. This means that there exist $N_1$ and $N_2$ so that
\begin{eqnarray}\label{V3}
n_k\geq N_1\Rightarrow\;\|x_{mn_k+i_0-1}-x_{mn_k+i_0}\|\leq d(A_{i_0},A_{i_0+1})+\varepsilon,
\end{eqnarray}
\begin{eqnarray}\label{V4}
n_k\geq N_2\Rightarrow\;\|x_{mn_k+m+i_0-1}-x_{mn_k+i_0}\|\leq d(A_{i_0},A_{i_0+1})+\varepsilon.
\end{eqnarray}
Let $N:=\max\{N_1,N_2\}$ and $n_k\geq N$. Thus from \eqref{V2}-\eqref{V4} and uniform convexity of $X$ we have that
\begin{eqnarray*}
\left\|{{x_{mn_k+i_0-1}+x_{m(n_k+1)+i_0-1} }\over{2}}-x_{mn_k+i_0}\right\|&\leq&\left(1-\delta\left({{\varepsilon_0}\over{d(A_{i_0},A_{i_0+1})+\varepsilon}}\right)\right)\times\\&&(d(A_{i_0},A_{i_0+1})+\varepsilon)
\end{eqnarray*}
because of the fact that $A_{i_0}$ is convex. According to the choice of $\varepsilon$ and the fact that $\delta$ is strictly increasing we can deduce that
\begin{eqnarray*}
\left\|{{x_{mn_k+i_0-1}+x_{m(n_k+1)+i_0-1} }\over{2}}-x_{mn_k+i_0}\right\|< d(A_{i_0},A_{i_0+1})
\end{eqnarray*} for $n_k\geq N$ which is a contradiction due to the definition of $d(A_{i_0},A_{i_0+1})$.
\end{proof}
\begin{remark}\label{PO1}In the cases when $1\leq p<\infty$,  and $A_i$s are convex subsets with odd indices, if $T$ is a cyclic $p$-$\varphi$-contraction mapping, by the same method as Proposition \ref{VV},we could show that $\|x_{mn+i}-x_{mn+m+i}\|\to0$ holds true for all $i\in\{0,\cdots,m-1\}$.
\end{remark}
\begin{theorem} Let $A_i$s be convex subsets with odd indices. Then for each $\varepsilon>0$, there exists a positive integer $N_0$ so that for all $n,k\geq N_0$
$$d_p(x_{mn},x_{mn+1},\cdots,x_{mn+m-1};x_{mk},x_{mk+1},\cdots,x_{mk+m-1})<d_p(A_1,\cdots,A_m)+\varepsilon.$$
\end{theorem}
\begin{proof}Assume that the conclusion of the theorem is not true, meaning that there exists an $\varepsilon_0>0$ so that for each positive integer $t$ there are $n_t\geq k_t\geq t$ satisfying
\begin{eqnarray*}d_p(x_{mn_t},x_{mn_t+1},\cdots,x_{mn_t+m-1};x_{mk_t},x_{mk_t+1},\cdots,x_{mk_t+m-1})\\\geq d_p(A_1,\cdots,A_m)+\varepsilon_0\qquad\qquad\qquad\qquad\qquad\end{eqnarray*} and
\begin{eqnarray*}d_p(x_{mn_t-m},x_{mn_t-m+1},\cdots,x_{mn_t-1};x_{mk_t},x_{mk_t+1},\cdots,x_{mk_t+m-1})\\<d_p(A_1,\cdots,A_m)+\varepsilon_0.\qquad\qquad\qquad\qquad\qquad\end{eqnarray*}
Hence
\begin{eqnarray*}d_p(A_1,\cdots,A_m)+\varepsilon_0&\leq& d_p(x_{mn_t},x_{mn_t+1},\cdots,x_{mn_t+m-1};x_{mk_t},x_{mk_t+1},\cdots,x_{mk_t+m-1})\\&=&\left(\sum_{i=1}^m\|x_{mn_t+i}-x_{mk_t+i+1}\|^p\right)^{1\over p}\\&\leq&\left(\sum_{i=1}^m\|x_{mn_t+i}-x_{mn_t-m+i}\|^p\right)^{1\over p}\\&&+\left(\sum_{i=1}^m\|x_{mn_t-m+i}-x_{mk_t+i+1}\|^p\right)^{1\over p}\\&\leq&\left(\sum_{i=1}^m\|x_{mn_t+i}-x_{mn_t-m+i}\|^p\right)^{1\over p}+d_p(A_1,\cdots,A_m)+\varepsilon_0.
\end{eqnarray*}
Letting $t\to\infty$ and using Proposition \ref{VV} and Remark \ref{PO1} we have that  $$\left(\sum_{i=1}^m\|x_{mn_t+i}-x_{mn_t-m+i}\|^p\right)^{1\over p}\to0.$$ Therefore,
\begin{eqnarray}\label{SD}
d_p(x_{mn_t},x_{mn_t+1},\cdots,x_{mn_t+m-1};x_{mk_t},x_{mk_t+1},\cdots,x_{mk_t+m-1})\nonumber\\\to d_p(A_1,\cdots,A_m)+\varepsilon_0.\qquad\qquad\qquad\qquad\qquad
\end{eqnarray}
On the other hand
%\end{eqnarray*}\begin{eqnarray*}
\begin{eqnarray*}
\left(\sum_{i=1}^m\|x_{mn_t+i}-x_{mk_t+i+1}\|^p\right)^{1\over p}&\leq&\left(\sum_{i=1}^m\|x_{mn_t+i}-x_{mn_t+m+i}\|^p\right)^{1\over p}\\&&+\left(\sum_{i=1}^m\|x_{mn_t+m+i}-x_{mk_t+m+i+1}\|^p\right)^{1\over p}\\&&+\left(\sum_{i=1}^m\|x_{mk_t+m+i+1}-x_{mk_t+i+1}\|^p\right)^{1\over p}\end{eqnarray*}\begin{eqnarray*}&\leq&\left(\sum_{i=1}^m\|x_{mn_t+i}-x_{mn_t+m+i}\|^p\right)^{1\over p}\\&&+\left(\sum_{i=1}^m\|x_{mn_t+1+i}-x_{mk_t+1+i+1}\|^p\right)^{1\over p}\\&&+\left(\sum_{i=1}^m\|x_{mk_t+m+i+1}-x_{mk_t+i+1}\|^p\right)^{1\over p}\\&\leq&\left(\sum_{i=1}^m\|x_{mn_t+i}-x_{mn_t+m+i}\|^p\right)^{1\over p}\\&&+\left(\sum_{i=1}^m\|x_{mn_t+i}-x_{mk_t+i+1}\|^p\right)^{1\over p}\\&&-\varphi\left(\left(\sum_{i=1}^m\|x_{mn_t+i}-x_{mk_t+i+1}\|^p\right)^{1\over p}\right)\\&&+\varphi(d_p(A_1,\cdots,A_m))\\&&+\left(\sum_{i=1}^m\|x_{mk_t+m+i+1}-x_{mk_t+i+1}\|^p\right)^{1\over p}\\&\leq&\left(\sum_{i=1}^m\|x_{mn_t+i}-x_{mn_t+m+i}\|^p\right)^{1\over p}\\&&+\left(\sum_{i=1}^m\|x_{mn_t+i}-x_{mk_t+i+1}\|^p\right)^{1\over p}\\&&+\left(\sum_{i=1}^m\|x_{mk_t+m+i+1}-x_{mk_t+i+1}\|^p\right)^{1\over p}.
\end{eqnarray*}
Now, letting $t\to\infty$, we come to
\begin{eqnarray*}
d_p(A_1,\cdots,A_m)+\varepsilon_0&\leq&d_p(A_1,\cdots,A_m)+\varepsilon_0\\&&-\lim_{t\to\infty}\varphi\left(\left(\sum_{i=1}^m\|x_{mn_t+i}-x_{mk_t+i+1}\|^p\right)^{1\over p}\right)\\&&+\varphi(d_p(A_1,\cdots,A_m))\\&\leq&d_p(A_1,\cdots,A_m)+\varepsilon_0.
\end{eqnarray*}
Hence
\begin{eqnarray*}
\lim_{t\to\infty}\varphi\left(\left(\sum_{i=1}^m\|x_{mn_t+i}-x_{mk_t+i+1}\|^p\right)^{1\over p}\right)=\varphi(d_p(A_1,\cdots,A_m)).
\end{eqnarray*}
But $\varphi$ is strictly increasing, so by \eqref{SD} we have that
\begin{eqnarray*}
\varphi(d_p(A_1,\cdots,A_m)+\varepsilon_0)&\leq&\lim_{t\to\infty}\varphi\left(\left(\sum_{i=1}^m\|x_{mn_t+i}-x_{mk_t+i+1}\|^p\right)^{1\over p}\right)\\&=&\varphi(d_p(A_1,\cdots,A_m))\\&<&\varphi(d_p(A_1,\cdots,A_m)+\varepsilon_0)
\end{eqnarray*}
which is impossible.
\end{proof}
Note that in the recent theorem if $p=\infty$, then it is sufficient that only one of $A_i$s is convex.
\begin{corollary}\label{V1}Let $A_i$s with odd indices be convex subsets. Then

1)  If $p=\infty$, and if $d(A_{i_0},A_{i_0+1})=d_\infty(A_1,\cdots,A_m)$ then for $\varepsilon>0$ there exists a positive integer $N$ so that for any $n\geq k\geq N$ we have that $\|x_{mn+i_0-1}-x_{mk+i_0}\|<d(A_{i_0},A_{i_0+1})+\varepsilon$.

2) If $1\leq p<\infty$, then for $\varepsilon>0$ and for any $1\leq i\leq m$  there exists $N_i$ so that for any $n\geq k\geq N_i$ we have that $\|x_{mn+i-1}-x_{mk+i}\|<d(A_i,A_{i+1})+\varepsilon$.
\end{corollary}
\begin{proof}
1) Let $\varepsilon>0$. According to the previous theorem there exists a positive integer $N$ so that for $n\geq k\geq N$ we have that $$d_\infty(x_{mn},x_{mn+1},\cdots,x_{mn+m-1};x_{mk},x_{mk+1},\cdots,x_{mk+m-1})<d(A_{i_0},A_{i_0+1})+\varepsilon.$$ Now the result is obvious because
$$\|x_{mn+i_0-1}-x_{mk+i_0}\|\leq d_\infty(x_{mn},x_{mn+1},\cdots,x_{mn+m-1};x_{mk},x_{mk+1},\cdots,x_{mk+m-1}).$$

2) Assume that 2) is not true and there exists some $i\in\{1,\cdots,m\}$ and $\varepsilon_0>0$ so that for any positive integer $N$ there exist $n_t\geq k_t\geq N$ and $\|x_{mn_t+i-1},x_{mk_t+i}\|\geq d(A_i,A_{i+1})+\varepsilon_0$. Let $\delta$ be a positive number with
\begin{eqnarray*}
0<\delta&<&\{d(A_1,,A_2)^p+\cdots+(d(A_{i},A_{i+1})+\varepsilon_0)^p+\cdots+d(A_{m-1},A_m)^p\}^{1\over p}\\&&-\{d(A_1,,A_2)^p+\cdots+d(A_{i},A_{i+1})^p+\cdots+d(A_{m-1},A_m)^p\}^{1\over p}.
\end{eqnarray*}
Now by the previous theorem we can find a positive integer $N_0$ such that for all $n\geq k\geq N_0$ we have that
$$d_p(x_{mn},x_{mn+1},\cdots,x_{mn+m-1};x_{mk},x_{mk+1},\cdots,x_{mk+m-1})-d_p(A_1,\cdots,A_m)<\delta.$$
But this is a contadiction because if $n_t\geq k_t\geq N_0$ with $$\|x_{mn_t+i-1}-x_{mk_t+i}\|\geq d(A_i,A_{i+1})+\varepsilon_0$$ then
\begin{eqnarray*}
0<\delta&<&\{d(A_1,,A_2)^p+\cdots+(d(A_{i},A_{i+1})+\varepsilon_0)^p+\cdots+d(A_{m-1},A_m)^p\}^{1\over p}\\&&-\{d(A_1,,A_2)^p+\cdots+d(A_{i},A_{i+1})^p+\cdots+d(A_{m-1},A_m)^p\}^{1\over p}\\
&\leq&\{\|x_{mn_t}-x_{mk_t+1}\|^p+\cdots+\|x_{mn_t+i-1}-x_{mk_t+i}\|^p+\cdots\\&&+\|x_{mn_t+m-1}-x_{mk_t}\|^p\}^{1\over p}\\&&-\{d(A_1,,A_2)^p+\cdots+d(A_{i},A_{i+1})^p+\cdots+d(A_{m-1},A_m)^p\}^{1\over p}<\delta.
\end{eqnarray*}
\end{proof}
\begin{theorem}\label{FP1}
Let $X$ be a uniformly convex Banach space and let $A_{i_0}$ be a closed convex subset of $X$ for some $i_0\in\{1,\cdots,m\}$. If $d_p(A_1,\cdots,A_m)=0$, then $T$ has a unique fixed point $x\in\cap_{i=1}^mA_i$ and $x_n\to x$ as $n\to\infty$.
\end{theorem}
\begin{proof}Let $\varepsilon>0$. By Proposition \ref{VV} or Remark \ref{PO1} there exists $N_1$ such that
$\|x_{mn+i_0-1}-x_{mn+i_0}\|<\varepsilon$ holds  true for all $n\geq N_1$. By Corollary \ref{V1} there exists $N_2$ so that $\|x_{mk+i_0-1}-x_{mn+i_0}\|<\varepsilon$  holds  true for all $k\geq N_2$. Thus if $N=\min\{N_1,N_2\}$ and $n>k\geq N$, Then
$$\|x_{mn+i_0-1}-x_{mk+i_0-1}\|\leq\|x_{mk+i_0-1}-x_{mn+i_0}\|+\|x_{mn+i_0-1}-x_{mn+i_0}\|<\varepsilon.$$Hence $\{x_{mn+i_0-1}\}$ is a Cauchy sequence in $A_{i_0}$ and so it is convergent to an element of $A_{i_0}$, $x$ say. Thus
by Theorem \ref{SS1} we have that $$d_p(x,Tx,\cdots,T^{m-1}x)=d_p(A_1,A_2,\cdots,A_m)=0$$ which means $$x=Tx=T^2x=\cdots=T^{m-1}x.$$
It follows that $x$ is a fixed point for $T$ and $x\in\cap_{i=1}^mA_i\neq\emptyset$. Now, let $y$ be another fixed point for $T$. Therefore $$\sum_{i=1}^m\|x-y\|=\sum_{i=1}^m\|Tx-Ty\|\leq\sum_{i=1}^m\|x-y\|-\varphi(\sum_{i=1}^m\|x-y\|)+\varphi(0).$$ But $\varphi$ is strictly increasing which implies that $\varphi(0)<\varphi(\sum_{i=1}^m\|x-y\|)\leq\varphi(0)$; a contradiction. We note that $\|x_{mn+i_0-1}-x_{mn+i_0}\|\to d(A_{i_0},A_{i_0+1})=0$ and $x_{mn+i_0-1}\to x$, therefore $$\|x_{mn+i_0}-x\|\leq\|x_{mn+i_0-1}-x\|+\|x_{mn+i_0}-x_{mn+i_0-1}\|\to0$$ which implies that $x_{mn+i_0}\to x$. Using the same mathod we could show that $x_{mn+i}\to x$ for all $i\in\{1,\cdots,m\}$. This ensures that $x_n\to x$ as $n\to\infty$ and we are done.
\end{proof}
\begin{theorem}\label{FP7}
Let $X$ be a uniformly convex Banach space and let $A_{i_0}$ be a closed convex subset of $X$ for $i_0\in\{1,\cdots,m\}$ with $d(A_{i_0},A_{i_0+1})=d_\infty(A_1,\cdots,A_m)$. Then $\{x_{mn+i_0-1}\}$ is a Cauchy sequence.
\end{theorem}
\begin{proof}With a slight modifying the proof of \cite[Theorem 7]{ThSh}, we can prove this theorem. For the sake of completeness of the discussion, we state the proof.  The case of $d_p(A_1,\cdots,A_m)=0$ has just discussed in Theorem \ref{FP1}, so we assume that $d_p(A_1,\cdots,A_m)\neq0.$ Let $\{x_{mn+i_0-1}\}$ is not  Cauchy. Thus there exists $\varepsilon_0>0$ so that for each positive integer $t$ there exist $n_t\geq k_t\geq t$ such that
\begin{eqnarray}\label{FP4}\|x_{mn_t+m-1}-x_{mk_t+m-1}\|\geq\varepsilon_0.\end{eqnarray} Pick a $0<\gamma<1$ so that ${{\varepsilon_0}\over{\gamma}}>d(A_{i_0},A_{i_0+1}),$ and let $\varepsilon>0$ be chosen with
\begin{eqnarray*}
0<\varepsilon<\min\left\{{{\varepsilon_0}\over{\gamma}}-d(A_{i_0},A_{i_0+1}),{{d(A_{i_0},A_{i_0+1})\delta(\gamma)}\over{1-\delta(\gamma)}}\right\}.
\end{eqnarray*}
By Theorem \ref{XX}, there exists $N_1$ so that
\begin{eqnarray}\label{FP5}\|x_{mn_t+i-1}-x_{mn_t+i}\|<d(A_{i_0},A_{i_0+1})+\varepsilon,\end{eqnarray}
for all $n_t\geq N_1$ and by Corollary \ref{V1}, there exists $N_2$ so that
\begin{eqnarray}\label{FP6}\|x_{mk_t+i-1}-x_{mn_t+i}\|<d(A_{i_0},A_{i_0+1})+\varepsilon\end{eqnarray}
for all $n_t\geq k_t\geq N_2$. Now, let $n_t\geq k_t\geq N=\max\{N_1,N_2\}$. The uniform convexity of $X$ and \eqref{FP4},\eqref{FP5} and \eqref{FP6} bring us to
\begin{eqnarray*}
\left\|{{x_{mn_t+i_0-1}+x_{mk_t+i_0-1} }\over{2}}-x_{mn_t+i_0}\right\|&\leq&\left(1-\delta\left({{\varepsilon_0}\over{d(A_{i_0},A_{i_0+1})+\varepsilon}}\right)\right)\times\\&&(d(A_{i_0},A_{i_0+1})+\varepsilon).
\end{eqnarray*}According to the choice of $\varepsilon$ and the fact that $\delta$ is strictly increasing we can deduce that
\begin{eqnarray*}
\left\|{{x_{mn_t+i_0-1}+x_{mk_t+i_0-1} }\over{2}}-x_{mn_t+i_0}\right\|< d(A_{i_0},A_{i_0+1})
\end{eqnarray*} for $n_t\geq k_t\geq N$ which is a contradiction due to the definition of $d(A_{i_0},A_{i_0+1})$. Hence $\{x_{mn+i_0-1}\}$ is a Cauchy sequence.
\end{proof}
\begin{theorem}Let $X$ be a uniformly convex Banach space and let $A_{i_0}$ be a closed convex subset of $X$ for $i_0\in\{1,\cdots,m\}$ with $d(A_{i_0},A_{i_0+1})=d_\infty(A_1,\cdots,A_m)$. Then there exists a unique $x\in A_{i_0}$ so that $x_{mn+i_0-1}\to x$, $T^{m}x=x$ and $$d_p(x,Tx,\cdots,T^{m-1}x)=d_p(A_1,A_2,\cdots,A_m).$$
\end{theorem}
\begin{proof}
By Theorem \ref{FP7}, $\{x_{mn+i_0-1}\}$ is a Cauchy sequence in $A_{i_0}$, so it is convergent to a point of $A_{i_0}$, $x$ say. By Theorem \ref{SS1} $$d_p(x,Tx,\cdots,T^{m-1}x)=d_p(A_1,A_2,\cdots,A_m).$$ To see that $T^mx=x$, we note that
\begin{eqnarray*}d_p(A_1,A_2,\cdots,A_m)&\leq& d_p(T^mx,Tx,\cdots,T^{m-1}x)\\&\leq&d_p(T^{m-1}x,x,\cdots,T^{m-2}x)\\&=&d_p(A_1,A_2,\cdots,A_m).\end{eqnarray*}Therefore $d_p(T^mx,Tx,\cdots,T^{m-1}x)=d_p(A_1,A_2,\cdots,A_m)$. This implies that $\|x-Tx\|=\|T^mx-Tx\|=d(A_{i_0},A_{i_0+1}).$ Because if these are not true, for instance assume that $\|x-Tx\|\neq d(A_{i_0},A_{i_0+1}),$ then there exists an $\varepsilon_0>0$ so that $\|x-Tx\|\geq d(A_{i_0},A_{i_0+1})+\varepsilon_0$. Hence
\begin{eqnarray*}
0<\delta&<&\{d(A_1,,A_2)^p+\cdots+(d(A_{i_0},A_{i_0+1})+\varepsilon_0)^p+\cdots+d(A_{m-1},A_m)^p\}^{1\over p}\\&&-\{d(A_1,,A_2)^p+\cdots+d(A_{i},A_{i+1})^p+\cdots+d(A_{m-1},A_m)^p\}^{1\over p}\\
&\leq&\{\|x-Tx\|^p+\cdots+\|T^{m-1}x-x\|^p\}^{1\over p}\\&&-\{d(A_1,,A_2)^p+\cdots+d(A_{i},A_{i+1})^p+\cdots+d(A_{m-1},A_m)^p\}^{1\over p}=0
\end{eqnarray*} which is impossible, thus $\|x-Tx\|=d(A_{i_0},A_{i_0+1}).$ Similarly $\|T^mx-Tx\|=d(A_{i_0},A_{i_0+1}).$ Now, let $T^mx\neq x$. Thus using the strict convexity of $X$, we have that
$$0<\left\|{{T^mx+x}\over{2}}-Tx\right\|<d(A_{i_0},A_{i_0+1}),$$ that is a contradiction. Because $A_{i_0}$ is convex which implies that ${{T^mx+x}\over{2}}\in A_{i_0}$. This ensures $T^mx=x$. Finally, we are going to show the uniqueness of $x$. For, assume that there exists another $y\in A_{i_0}$, so that $T^my=y$ and
$$d_p(y,Ty,\cdots,T^{m-1}y)=d_p(A_1,A_2,\cdots,A_m).$$This implies that $\|y-Ty\|=d(A_{i_0},A_{i_0+1}).$
By \eqref{FP8}, we have that
\begin{eqnarray*}
d_p(y,Tx,T^2x,\cdots,T^{m-1}x;y,Tx,T^2y,T^3y,\cdots,T^{m-1}y)\leq\qquad\qquad\qquad\\ d_p(T^{m-1}y,x,Tx,\cdots,T^{m-2}x;T^{m-1}y,x,Ty,\cdots,T^{m-2}y)\leq\qquad\qquad\\ d_p(y,Tx,T^2x,\cdots,T^{m-1}x;y,Tx,T^2y,T^3y,\cdots,T^{m-1}y).\;\;\;\;\;\;\;
\end{eqnarray*}
Hence
\begin{eqnarray*}
d_p(y,Tx,T^2x,\cdots,T^{m-1}x;y,Tx,T^2y,T^3y,\cdots,T^{m-1}y)=\qquad\qquad\qquad\qquad\\ d_p(T^{m-1}y,x,Tx,\cdots,T^{m-2}x;T^{m-1}y,x,Ty,\cdots,T^{m-2}y).
\end{eqnarray*}
Now if $$d_p(y,Tx,T^2x,\cdots,T^{m-1}x;y,Tx,T^2y,T^3y,\cdots,T^{m-1}y)>d_p(A_1,\cdots,A_m),$$ then since $T$ is a $p$-$\varphi$-contraction with a strictly increasing function $\varphi$, we have that
\begin{eqnarray*}
d_p(y,Tx,T^2x,\cdots,T^{m-1}x;y,Tx,T^2y,T^3y,\cdots,T^{m-1}y)\leq\qquad\qquad\qquad\qquad\\ d_p(T^{m-1}y,x,Tx,\cdots,T^{m-2}x;T^{m-1}y,x,Ty,\cdots,T^{m-2}y)-\qquad\qquad\\\varphi(d_p(T^{m-1}y,x,Tx,\cdots,T^{m-2}x;T^{m-1}y,x,Ty,\cdots,T^{m-2}y))+\varphi(d_p(A_1,\cdots,A_m))<\\ d_p(T^{m-1}y,x,Tx,\cdots,T^{m-2}x;T^{m-1}y,x,Ty,\cdots,T^{m-2}y)-\qquad\qquad\\\varphi(d_p(A_1,\cdots,A_m))+\varphi(d_p(A_1,\cdots,A_m))=\\d_p(y,Tx,T^2x,\cdots,T^{m-1}x;y,Tx,T^2y,T^3y,\cdots,T^{m-1}y).\qquad\qquad
\end{eqnarray*}This is impossible so $$d_p(y,Tx,T^2x,\cdots,T^{m-1}x;y,Tx,T^2y,T^3y,\cdots,T^{m-1}y)=d_p(A_1,\cdots,A_m),$$ which implies that
$\|x-Ty\|=\|Tx-y\|=d(A_{i_0},A_{i_0+1})$. But it follows from the convexity of $A_{i_0}$ and the strictly convexity of $X$ that
$$0<\left\|{{y+x}\over{2}}-Tx\right\|=\left\|{{x-Ty}\over{2}}+{{y-Tx}\over{2}}\right\|<d(A_{i_0},A_{i_0+1}),$$ a contradiction. Thus $x=y$.
\end{proof}
\section{Reflexive Banach spaces}
In this section we study the $p$-$\alpha$-contraction mappings in reflexive Banach spaces. The method of this section is an improvement of that of \cite[Section 3]{ThSh}. In the rest of the paper, we assume that $X$ is a reflexive Banach space, $A_i$s are non-empty subsets of $X$, and $T:\cup_{i=1}^mA_i\to\cup_{i=1}^mA_i$ is a cyclic $p$-$\alpha$-contraction, for some positive real number $\alpha$ with $0<\alpha^m<{1\over{2^{1\over p}}}$.
\begin{theorem}\label{REFL}
Suppose $A_i$s are weakly closed subsets of $X$. Then, there exists $(\hat{x}_1,\hat{x}_2,\cdots,\hat{x}_m)\in A_1\times A_2\times\cdots\times A_m$ such that
$$d_p(\hat{x}_1,\hat{x}_2,\cdots,\hat{x}_m)=d_p(A_1,\cdots,A_m).$$
\end{theorem}
\begin{proof}In the case when $d_p(A_1,\cdots,A_m)=0$ the conclusion is obtained from Theorem \ref{BANACH}. So assume that $d_p(A_1,\cdots,A_m)\neq0$. By Theorem \ref{MET1} sequences $\{x_{mn+i-1}\}$s are all bounded. Thus these sequences have weakly convergent subsequences. Therefore we could fined the sequence $\{n_k\}$ of positive integers so that $\{x_{mn_k+i-1}\}$s are all convergent. Assume that $x_{mn_k+i-1}\xrightarrow{w}\hat{x}_i$, with each $\hat{x}_i$ in $A_i$, where $\xrightarrow{w}$ stands for weakly convergence. Assume that $1\leq p<\infty$. It follows that $x_{mn_k+i-1}-x_{mn_k+i}\xrightarrow{w}\hat{x}_i-\hat{x}_{i+1}\neq0$. This means that there exist real bounded linear functionals $f_i$ so that $\|f_i\|=1$ and $f_i(\hat{x}_i-\hat{x}_{i+1})=\|\hat{x}_i-\hat{x}_{i+1}\|.$ Note that $\hat{x}_{m+1}=\hat{x}_{1}$. $$\lim f_i(x_{mn_k+i-1}-x_{mn_k+i})=f_i(\hat{x}_i-\hat{x}_{i+1})=\|\hat{x}_i-\hat{x}_{i+1}\|.$$ On the other hand $$|f_i(x_{mn_k+i-1}-x_{mn_k+i})|\leq \|x_{mn_k+i-1}-x_{mn_k+i}\|.$$ Hence by Lemma \ref{V1} we have that $$\|\hat{x}_i-\hat{x}_{i+1}\|=\lim f_i(x_{mn_k+i-1}-x_{mn_k+i})\leq \lim\|x_{mn_k+i-1}-x_{mn_k+i}\|=d(A_i,A_{i+1}).$$ Thus $\|\hat{x}_i-\hat{x}_{i+1}\|=d(A_i,A_{i+1})$ for $i=1,\cdots,m$ and this means $$d_p(\hat{x}_1,\hat{x}_2,\cdots,\hat{x}_m)=d_p(A_1,\cdots,A_m).$$ For the case when $p=\infty$ the proof is the same.
\end{proof}
\begin{theorem}\label{REFL3}
Let $A_{i_0}$ is weakly closed for some $i_0\in\{1,\cdots,m\}$. Then there exists $x\in A_{i_0}$ so that $d_p(x,Tx,\cdots,T^{m-1}x)=d_p(A_1,\cdots,A_m)$ provided that $T$ is wakly continuous on $A_{i_0}$.
\end{theorem}
\begin{proof}In the case when $d_p(A_1,\cdots,A_m)=0$ the result is obtained from Theorem \ref{BANACH}. So assume that $d_p(A_1,\cdots,A_m)\neq0$. By Theorem \ref{MET1} sequence $\{x_{mn+i_0-1}\}$ is bounded. Thus it has a weakly convergent subsequence $\{x_{mn_k+i_0-1}\}$ weakly converging to an element of $A_{i_0}$, $x$ say.
$T$ is wakly continuous on $A_{i_0}$, so are $T^j$. This means $x_{mn_k+i_0+j-1}\xrightarrow{w}T^jx$ as $k\to\infty$ for $j\in\{1,\cdots,m\}$. Thus  $x_{mn_k+i_0+j-1}-x_{mn_k+i_0+j}\xrightarrow{w}T^jx-T^{j+1}x\neq0$. Using the same method as Theorem \ref{REFL} we could show that $d_p(x,Tx,\cdots,T^{m-1}x)=d_p(A_1,\cdots,A_m)$ as we please.
\end{proof}
\begin{theorem}\label{REFL5}
 Let $A_i$s be closed and convex subsets of a reflexive and strictly convex Banach space $X$ and let $(A_i-A_i)\cap(A_{i+1}-A_{i+1})=\{0\}$, for all $i\in\{1,\cdots,m\}$. Then, there exists a unique $x$ in $A_1$, so that $T^mx=x$ and $$d_p(x,Tx,T^2x,\cdots,T^{m-1}x)=d_p(A_1,\cdots,A_{m}).$$
\end{theorem}
\begin{proof}In the case when $d_p(A_1,\cdots,A_m)=0$ the conclusion is obtained from Theorem \ref{BANACH}. So assume that $d_p(A_1,\cdots,A_m)\neq0$. Since $A_i$s are closed and convex subsets, they are weakly closed subsets of $X$. Thus according to Theorem \ref{REFL} there exists $(\hat{x}_1,\hat{x}_2,\cdots,\hat{x}_m)\in A_1\times A_2\times\cdots\times A_m$ such that
$$d_p(\hat{x}_1,\hat{x}_2,\cdots,\hat{x}_m)=d_p(A_1,\cdots,A_m).$$ To show that $(\hat{x}_1,\hat{x}_2,\cdots,\hat{x}_m)$ is unique, assume that there exists another $(\hat{y}_1,\hat{y}_2,\cdots,\hat{y}_m)\in A_1\times A_2\times\cdots\times A_m$ such that
$$d_p(\hat{y}_1,\hat{y}_2,\cdots,\hat{y}_m)=d_p(A_1,\cdots,A_m),$$ and assume that  $\hat{x}_{i_0}\neq\hat{y}_{i_0}$ for some $i_0$. Obviously $d(\hat{y}_{i},\hat{y}_{i+1})=d(A_{i},A_{i+1})$ for all $i\in\{1,\cdots,m\}$. Therefore, as $(A_{i_0}-A_{i_0})\cap(A_{i_0+1}-A_{i_0+1})=\{0\}$, we conclude that $\hat{x}_{i_0}-\hat{y}_{i_0}\neq\hat{x}_{i_0}-\hat{y}_{i_0}$. Since $A_{i_0}$ and $A_{i_0+1}$ are both convex subsets and $X$ is strictly convex, we have that
$$\left\|{{\hat{x}_{i_0}+\hat{y}_{i_0}}\over{2}}-{{\hat{x}_{i_0+1}+\hat{y}_{i_0+1}}\over{2}}\right\|=\left\|{{\hat{x}_{i_0}-\hat{x}_{i_0+1}}\over{2}}+{{\hat{y}_{i_0}-\hat{y}_{i_0+1}}\over{2}}\right\|<d(A_{i_0},A_{i_0+1}),$$ which is a  contradiction. Thus $\hat{x}_{i}=\hat{y}_{i}$ for all $i$ and consequently $(\hat{x}_1,\hat{x}_2,\cdots,\hat{x}_m)=(\hat{y}_1,\hat{y}_2,\cdots,\hat{y}_m)$. As
\begin{eqnarray*}
d_p(A_1,\cdots,A_m)&\leq& d_p(T^m\hat{x}_1,T^m\hat{x}_2,\cdots,T^m\hat{x}_m)\\&\leq& d_p(\hat{x}_1,\hat{x}_2,\cdots,\hat{x}_m)\\&=&d_p(A_1,\cdots,A_m),
\end{eqnarray*}
 we infer that $$d_p(T^m\hat{x}_1,T^m\hat{x}_2,\cdots,T^m\hat{x}_m)=d_p(\hat{x}_1,\hat{x}_2,\cdots,\hat{x}_m)=d_p(A_1,\cdots,A_m).$$ The uniqueness of $(\hat{x}_1,\hat{x}_2,\cdots,\hat{x}_m)$ implies that $T^m\hat{x}_i=\hat{x}_i$ for any $i$. The same argument shows that $T\hat{x}_{i}=\hat{x}_{i+1}$. Thus $\hat{x}_2=T\hat{x}_1$, $\hat{x}_3=T\hat{x}_2=T^2\hat{x}_1$,$\cdots$, $\hat{x}_m=T^{m-1}\hat{x}_1$, and we are done.
\end{proof}
\begin{theorem}Let $A_{i_0}$ be a closed and convex subset of a reflexive and strictly convex Banach space $X$, and let $T$ is weakly continuous on $A_{i_0}$. Then, there exists a unique $x$ in $A_1$, so that $T^mx=x$ and $$d_p(x,Tx,T^2x,\cdots,T^{m-1}x)=d_p(A_1,\cdots,A_{m}).$$
\end{theorem}
\begin{proof}In the case when $d_p(A_1,\cdots,A_m)=0$ the conclusion is obtained from Theorem \ref{BANACH}. So assume that $d_p(A_1,\cdots,A_m)\neq0$. Since $A_{i_0}$ is a closed and convex subset, it is a weakly closed subset of $X$. Thus according to Theorem \ref{REFL3} there exists $x\in A_{i_0}$ so that $d_p(x,Tx,\cdots,T^{m-1}x)=d_p(A_1,\cdots,A_m)$. If $T^mx\neq x$, then $T^mx-Tx\neq x-Tx$. It follows from the convexity of $A_{i_0}$ and the strictly convexity of $X$ that $$0<\left\|{{{T^mx}+x}\over{2}}-Tx\right\|=\left\|{{{T^mx}-Tx}\over{2}}+{{x-Tx}\over{2}}\right\|<d(A_{i_0},A_{i_0+1}),$$ which is impossible. Thus $T^mx=x$. The proof of uniqueness of $x$ is exactly the same as that of Theorem \ref{REFL5}.
\end{proof}
\bibliographystyle{amsplain}

\begin{thebibliography}{10}
\bibitem{ThSh}
 M. A. Al-Thagafi and Naseer Shahzad
 \textit{ Convergence and existence results for best proximity points},
Nonlinear Analysis;TMA,\textbf{70}(2009)3665--3671.


\bibitem{ElVe}
A.A. Eldred and P. Veeramani, \textit{ Existence and convergence of best proximity points}, J. Math. Anal. Appl., \textbf{323} (2006), 1001--1006.

\bibitem{Ki}
W. A. Kirk \textit{ Contraction mappings and extensions, in: W.A. Kirk, B. Sims (Eds)}, Handbook of Metric Fixed Point Theory, Kluwer Academic Publishers, Dordrecht,2001,pp.1--34.

\bibitem{KiSrVe}
W. A. Kirk, P. S. Srinivasan and P. Veeramani, \textit{ Fixed points for mappings satisfying cyclical contractive conditions}, Fixed Point Theory, \textbf{4} (2003), 79--89.

\end{thebibliography}

\end{document}